\documentclass[12pt]{amsart}

\usepackage{amsmath,color,amssymb}
\usepackage{amsthm}
\usepackage{hyperref}
\usepackage[margin=1.0in]{geometry}
\usepackage{tikz,tikz-cd}

\numberwithin{equation}{section}

\newtheorem{prop}{Proposition}
\newtheorem{lemma}[prop]{Lemma}

\newtheorem{thm}[prop]{Theorem}
\newtheorem{cor}[prop]{Corollary}

\numberwithin{prop}{section}

\theoremstyle{definition}
\newtheorem{defn}[prop]{Definition}

\newtheorem{rmk}[prop]{Remark}

\renewcommand{\geq}{\geqslant}
\renewcommand{\leq}{\leqslant}

\newcommand{\del}{\partial}
\newcommand{\dt}{\frac{\partial}{\partial t}}
\newcommand{\brs}[1]{\left| #1 \right|}

\newcommand{\gG}{\Gamma}

\newcommand{\gD}{\Delta}
\newcommand{\gd}{\delta}

\newcommand{\gk}{\kappa}
\newcommand{\gl}{\lambda}

\newcommand{\gw}{\omega}
\newcommand{\ga}{\alpha}
\newcommand{\gb}{\beta}
\newcommand{\gL}{\Lambda}
\renewcommand{\ge}{\epsilon}
\newcommand{\N}{\nabla}
\newcommand{\FF}{\mathcal F}
\newcommand{\NN}{\mathcal N}
\newcommand{\WW}{\mathcal W}

\newcommand{\til}[1]{\widetilde{#1}}

\renewcommand{\bar}[1]{\overline{#1}}

\newcommand{\IP}[1]{\left<#1\right>}

\DeclareMathOperator{\Rc}{Rc}
\DeclareMathOperator{\Rm}{Rm}

\DeclareMathOperator{\divg}{div}

\DeclareMathOperator{\Vol}{Vol}

\DeclareMathOperator{\Area}{Area}

\begin{document}

\title[Scalar curvature, entropy, and generalized Ricci flow]{Scalar curvature, entropy, and generalized Ricci flow}

\begin{abstract} We derive a family of weighted scalar curvature monotonicity formulas for generalized Ricci flow, involving an auxiliary dilaton field evolving by a certain reaction-diffusion equation motivated by renormalization group flow.  These scalar curvature monotonicities are dual to a new family of Perelman-type energy and entropy monotonicity formulas by coupling to a solution of the associated weighted conjugate heat equation.  In the setting of Ricci flow, we further obtain a new family of convex Nash entropies and pseudolocality principles.
\end{abstract}

\author{Jeffrey Streets}
\address{Rowland Hall\\
         University of California, Irvine, CA}
\email{\href{mailto:jstreets@uci.edu}{jstreets@uci.edu}}

\date{\today}

\maketitle

\section{Introduction}

A dominant theme in the analysis of Ricci flow is the understanding of curvature positivity conditions preserved by the flow \cite{Wilking, BS, Hamilton3folds, Ham4PCO, Ham4PIC}.  Most fundamental among these, as observed by Hamilton in his original paper \cite{Hamilton3folds}, is the preservation of a lower bound on the scalar curvature.  This bound is essential for detailed analyses of heat kernels, ancient solutions, and singularity formation of Ricci flow (cf. e.g. \cite{BamlerZhang, ChenPSC, CMSing, HeinNaber, Zhang1, Zhang2, Zhang3}).  A second dominant theme is the key role played by self-similar solutions of the flow, i.e. Ricci solitons, which partly indicate the subtle interplay between Ricci flow and the diffeomorphism group.  Such solutions, and the interaction between Ricci flow and the diffeomorphism group, lie at the foundation of various key estimates for Ricci flow, such as Hamilton's Harnack estimate \cite{HamHarnack}, and Perelman's energy, entropy, reduced volume functionals, and differential Harnack estimate \cite{Perelman1}.  These various tools combine to reveal the structure of singular sets of Ricci flow \cite{Perelman1, Bamler2, Bamler1, Bamler3}, leading to  deep topological applications \cite{Perelman2, BK}.  In this paper we extend the fundamental circle of ideas around scalar curvature monotonicity,  Harnack estimates, and heat kernels to \emph{generalized Ricci flow}.  In the process we obtain several new estimates for Ricci flow, in particular a new family of convex Nash entropies, and pseudolocality estimates

  Given a smooth manifold $M$, a one-parameter family of metrics $g_t$ and closed three-forms $H_t$ is a solution of generalized Ricci flow if
\begin{align*}
\dt g =&\ - 2 \Rc + \tfrac{1}{2} H^2,\\
\dt b =&\ - d^*_g H, \qquad H = H_0 + db,
\end{align*}
where $H^2(X,Y) = \IP{i_X H, i_Y H}$.  This equation arises independently in mathematical physics \cite{OSW,Polchinski}, complex geometry \cite{PCF, PCFReg}, and generalized geometry \cite{GF19, StreetsTdual, GKRF}, and we refer to \cite{GRFbook} for further background.  Some global existence and convergence results can be found in \cite{ASnondeg, JFS, SRYM2, StreetsND}.  Note that $H \equiv 0$ is preserved by the flow (cf. \cite{GRFbook} Proposition 4.20), and the metric then solves Ricci flow.  Thus in the remainder of this paper many results are for generalized Ricci flow, with the attendant results for Ricci flow occurring as a special case.

\subsection{Scalar curvature monotonicity}

Given a metric $g$, closed three-form $H$, and smooth function $f$, let
\begin{align*}
\Rc^{H,f} = \Rc - \tfrac{1}{4} H^2 +\N^2 f - \tfrac{1}{2} \left( d^*_g H + i_{\N f} H \right), \qquad R^{H,f} := R - \tfrac{1}{12} \brs{H}^2 + 2 \gD f - \brs{\N f}^2.
\end{align*}
The tensor $\Rc^{H,f}$ reduces to the Ricci tensor of the Bismut connection with torsion $H$ when $f = 0$, and in general can be motivated by extending ideas from Bakry-Emery \cite{BakryEmery} to the Laplacian of the Bismut connection acting on one-forms.  The scalar curvature $R^{H,f}$ arises in the Lichnerowicz-type formula for the cubic Dirac operator of Bismut \cite{Bismut} in the case $f = 0$.  The general case occurs when computing this formula using a weighted volume form (cf. \cite{Ozuch, Perelman1}).

Perelman's energy and entropy monotonicity formulas can be interpreted as differential inequalities for the weighted scalar curvature $R^f$.  The key point is to allow the weight $u = e^{-f}$ to evolve by the conjugate heat equation.  In this circle of ideas the function $u = e^{-f}$ is the ``dilaton'' in physics terminology.  Interestingly, in mathematical physics literature a different equation is suggested for the dilaton in the RG flow of the $H$-twisted nonlinear sigma model \cite{Polchinski}.  Specifically, given $(g_t, H_t)$ a solution of generalized Ricci flow, we let $\square = \dt - \Delta$ denote the forward heat operator, and fix $\phi$ a solution to the associated \emph{dilaton flow}, 
\begin{align*}
\square \phi = \tfrac{1}{6} \brs{H}^2.
\end{align*}
With the setup above we then obtain the following evolution equation for the generalized scalar curvature:
\begin{prop} \label{p:scalarmon} (cf. Proposition \ref{p:scalarev}) Given $(M^n, g_t, H_t, \phi_t)$ a solution to generalized Ricci flow, one has
\begin{align} \label{f:RFscal}
\square R^{H,\phi} = 2 \brs{\Rc^{H,\phi}}^2.
\end{align}
\end{prop}

\begin{rmk} In the case of Ricci flow the dilaton flow is simply the forward heat flow.  Here the weighted scalar curvature monotonicity appears in (\cite{ChowRFbook} Chapter 7 Lemma 6.88).  There monotonicty formulas are shown for a one-parameter family of dilaton flows interpolating between the forward heat flow and the conjugate heat equation, which incidentally also extend to generalized Ricci flow.
\end{rmk}

Proposition \ref{p:scalarmon} implies that a lower bound on $R^{H,\phi}$ is preserved on a compact manifold.  Also by applying the strong maximum principle we obtain a rigidity result.
\begin{cor} \label{c:rigidity} (cf. Corollary \ref{c:rigidity2}) Let $(M, g, H, \phi)$ satisfy $R^{H,\phi} \geq 0$.  If $M$ is compact, then either
\begin{enumerate}
\item The triple $(g, H, \phi)$ defines a generalized Ricci soliton and $R^{H,\phi} \equiv 0$
\item The manifold $M$ admits a triple $(\bar{g}, \bar{H}, \bar{\phi})$ such that $R^{\bar{H},\bar{\phi}} > 0$ everywhere.
\end{enumerate}
\end{cor}

\begin{rmk} The scalar curvature evolution equation of Proposition \ref{p:scalarmon} and attendant corollaries can be generalized in several ways.  In particular, noting that all quantities involved ultimately only depend on $d f$, one may replace $df$ with a general one-form $\ga$ evolving by the operator $\square \ga = \tfrac{1}{6} d \brs{H}^2$, and obtain a monotone curvature quantity (cf. \S \ref{s:oneform}).  Furthermore the tensor $H$ may be replaced with a formal linear combination of differential forms of all degrees, and with an appropriately weighted scalar curvature and dilaton flow one again obtains a monotone curvature quantity (cf. \S \ref{s:GGRF}).  One special case of this is the extended Ricci flow system of List \cite{List}, coupling to an exact $1$-form.  Another special case is the Ricci-Yang-Mills flow (\cite{StreetsThesis, YoungThesis, StreetsRYMsurfaces, SRYM2}), a coupling of the Ricci and Yang-Mills flows which in the case of abelian structure group corresponds to the case that $H$ is a two-form, specifically the principal curvature.
\end{rmk}

\subsection{Entropy formulas and Harnack estimates}

Further structure is revealed when we treat $e^{-\phi}$ as a volume density, mirroring the role played by $u = e^{-f}$ in Perelman's work.  
Perelman's idea underlying his energy and entropy formulas is to let $u = e^{-f}$ be a solution of the conjugate heat equation $\square^* u = 0$, where
\begin{align*}
\square^* = - \dt - \gD + R.
\end{align*}
Given this, one obtains the differential equation
\begin{align} \label{f:RFsHarnack}
\square^* \left( R^f u \right) =&\ - 2 \brs{\Rc^f}^2 u.
\end{align}
This is the pointwise computation underlying the monotonicity of Perelman's $\FF$-functional, and including a further weighting of $u$ by a time scale yields Perelman's entropy density monotonicity.  There is a curious duality between equations (\ref{f:RFscal}) (in the setting of Ricci flow) and (\ref{f:RFsHarnack}).  On the one hand, the weighted scalar curvature is a supersolution to a forward heat equation when the weight satisfies the forward heat equation.  On the other hand, after coupling to a solution of the conjugate heat equation, the weighted scalar curvature is a subsolution to the conjugate heat equation.  We next clarify and deepen this apparent linkage, generalizing the circle of ideas around scalar curvature, entropy formulas, and conjugate heat kernels.  

We return to the setting of generalized Ricci flow, where the conjugate heat operator takes the form
\begin{align*}
\square^* = - \dt - \gD + R - \tfrac{1}{4} \brs{H}^2.
\end{align*}
The fundamental formula (\cite{GRFbook} Chapter 6) underlying the energy monotonicity which generalizes (\ref{f:RFsHarnack}) is then
\begin{align} \label{f:GRFsHarnack}
\square^* \left( R^{H,f} u \right) =&\ - 2 \brs{\Rc^{H,f}}^2 u.
\end{align}
When integrated against $e^{-f} dV_g$, this pointwise formula yields the gradient flow interpretation for generalized Ricci flow.  As natural as this differential equation may seem, it is difficult to exploit in part because the conjugate heat operator itself is difficult to control.  In the setting of Ricci flow the a priori lower bound on the scalar curvature controls the reaction term in the conjugate heat operator, and this plays a key role in many applications.  

We get a better behaved heat kernel by including a further weight given by a solution to the dilaton flow.  Indeed the conjugate of the heat operator, taken with respect to the measure $dm = e^{-\phi} dV_g$, is
\begin{align*}
\square^*_{\phi} = - \dt - \gD + 2 \N \phi + R^{H,\phi}.
\end{align*}
We then obtain a generalization of (\ref{f:GRFsHarnack}):

\begin{prop} \label{p:steadyharnackev} (cf. Proposition \ref{p:steadyharnackev2}) Let $(M, g_t, H_t, \phi_t)$ denote a solution to generalized Ricci flow, and suppose $u = e^{-f}$ is a solution of $\square^*_{\phi} u = 0$.
Then
\begin{align*}
\square^*_{\phi} \left( R^{H,f+ \phi} u \right) =&\ - 2 \brs{\Rc^{H,f+\phi}}^2 u.
\end{align*}
\end{prop}
\noindent Note that in this formula there are naturally two functional degrees of freedom given by a solution to the dilaton flow, and then a solution to the weighted conjugate heat equation.  This furthermore yields a generalization of the gradient flow property of generalized Ricci flow (cf. Proposition \ref{p:gradientflow}).

By including further weighting of $u$ by a time scale we also obtain a generalization of Perelman's shrinker entropy monotonicity and differential Harnack estimate, parameterized by the choices of $f$ and $\phi$.  In particular, recall Perelman's entropy density
\begin{align*}
W^{H,F} = \tau R^{H,F} + F - n,
\end{align*}
To obtain a monotone quantity under generalized Ricci flow we must include a further functional parameter which measures the concentration of $H$, namely a solution to the \emph{conjugate dilaton flow}.  In particular we let $\psi$ denote a solution of
\begin{align*}
\square^*_{\phi} (\psi u) = - \tfrac{1}{6} \brs{H}^2 u.
\end{align*}
We then obtain:
\begin{prop} \label{p:GRFharnack} (cf. Proposition \ref{p:Ricciflowharnack}) Let $(M, g_t, H_t, \phi_t)$ denote a solution to generalized Ricci flow, and suppose $u = (4 \pi \tau)^{-\frac{n}{2}} e^{-f}$ is a solution of $\square^*_{\phi} u = 0$, where $\tau = T - t$ for some fixed $T$.  Furthermore suppose $\psi$ is a solution of the conjugate dilaton flow.  Then
\begin{align*}
\square^*_{\phi} \left[ (W^{H,f+\phi} + \psi) u\right] =&\ - 2 \tau \brs{\Rc^{H,f+\phi} - \frac{1}{2 \tau} g}^2 u.
\end{align*}
\end{prop}
\noindent This yields that if $u$ approaches a weighted Dirac delta at some forward time then
\begin{align*}
(W^{H,f + \phi} + \psi) u \leq 0,
\end{align*}
generalizing Perelman's Harnack estimate (cf. Corollary \ref{c:GRFHarnack}).  Due to the presence of the conjugate dilaton flow solution, finding further geometric applications requires a more detailed understanding of the torsion $H$.

\subsection{Further applications to Ricci flow}

To conclude we observe some formal extensions of some key results in the analysis of Ricci flow to the setting of weighted scalar curvature.  First we note an extension of the definition of the Nash and Perelman entropies, adapted to weighted scalar curvature.  In Proposition \ref{p:weightednash} we show convexity of the Nash entropy, extending fundamental observations in \cite{HeinNaber}.
Going further, recall a key application of Perelman's differential Harnack estimate is the pseudolocality estimate \cite{Perelman1}, which roughly says that almost Euclidean regions will regularize for a short time.  The strength of this result is that `almost Euclidean,' is measured in a very weak sense, namely by a lower scalar curvature bound and an almost-Euclidean isoperimetric inequality.  Based on the generalized entropy monotonicity formulas above, we give an extension of this result, involving the weighted scalar curvature and isoperimetric inequality (Theorem \ref{t:RFgenpseudo}).  The proof follows Perelman's original proof until the final stages, where the entropy integrand is manipulated to exploit the weighted scalar curvature bound.  In the end we require a technical result relating the weighted isoperimetric inequality to the weighted log-Sobolev inequality (Theorem \ref{t:isotosob}), proved using the method of Steiner symmetrization.

\begin{thm} \label{t:RFgenpseudo} (cf. Theorem \ref{t:RFgenpseudo2}) For every $\ga > 0$ there exist $\gd, \ge > 0$ satisfying the following: Suppose we have a smooth pointed Ricci flow solution $(M, (p_0, 0), g_t)$ defined for $t \in [0, (\ge r_0)^2]$, such that each time slice is complete.  Suppose that there exists $\phi_0 \in C^{\infty}_0(M)$ such that
\begin{enumerate}
\item $R^{\phi_0}(p, 0) \geq - r_0^{-2}$ for any $p \in B_0(p_0, r_0)$,
\item $\phi_0(p, 0) \geq - \gd$ for any $p \in B_0(p_0, r_0)$,
\item The $\phi_0$-weighted isoperimetric constant of $B_0(p_0, r_0)$ satisfies $I_n^{\phi_0} \geq (1 - \gd ) c_n$, where $c_n$ denotes the Euclidean isoperimetric constant.
\end{enumerate}
Then $\brs{\Rm}(p,t) < \ga t^{-1} + (\ge r_0)^{-2}$ whenever $0 < t \leq (\ge r_0)^2$ and $d_t(p,p_0) \leq \ge r_0$.
\end{thm}

\vskip 0.1in

\textbf{Acknowledgements:} This work emerged from work with V. Apostolov and Y. Ustinovskiy on scalar curvature in generalized K\"ahler geometry.  Companion work will appear on the role of scalar curvature in generalized K\"ahler geometry \cite{ASUScal}, and on the generalized K\"ahler Calabi-Yau conjecture \cite{AFSUCY}.  We further thank Ben Chow and Richard Bamler for helpful comments.

\section{Weighted scalar curvature monotonicity formulas} \label{s:GSC}

In this section we show several monotonicity formulas for weighted scalar curvatures and certain further generalizations.  First we prove Proposition \ref{p:scalarmon} of the introduction, and derive a rigidity result using the strong maximum principle.  Then we extend to a more general setting where $df$ is replaced by an arbitrary $1$-form, and then to the setting where the metric flow is coupled to the heat flow for differential forms of arbitrary degree.

\subsection{Weighted scalar curvature monotonicity and rigidity results}

To begin we formalize some definitions from the introduction.

\begin{defn} Given a smooth manifold $M$, a triple $(g, H, \phi)$ of a Riemannian metric, closed three-form $H$, and function $\phi$ determine a \emph{twisted Bakry-Emery curvature} 
\begin{align*}
\Rc^{H,\phi} = \Rc - \tfrac{1}{4} H^2 +\N^2 \phi - \tfrac{1}{2} \left( d^*_g H + i_{\N \phi} H \right).
\end{align*}
This data also determines a \emph{generalized scalar curvature}
\begin{align*}
R^{H,\phi} := R - \tfrac{1}{12} \brs{H}^2 + 2 \gD \phi - \brs{\N \phi}^2.
\end{align*}
\end{defn}

\begin{defn} Given $(M, g_t, H_t)$ a solution to generalized Ricci flow, a one-parameter family $\phi_t$ satisfies the \emph{dilaton flow} if
\begin{gather*}
\begin{split} 
\square \phi =&\ \tfrac{1}{6} \brs{H}^2.
\end{split}
\end{gather*}
\end{defn}

\begin{rmk} We separate the terminology of a solution to the dilaton flow associated to a solution of generalized Ricci flow to emphasize that the two flows are decoupled, and for instance the initial data for $\phi$ is arbitrary.  On the other hand, for convenience we will also refer to a triple $(g_t, H_t, \phi_t)$ as a solution of generalized Ricci flow, where a particular solution $\phi$ to the dilaton flow has been selected.
\end{rmk}

\begin{prop} (cf. Proposition \ref{p:scalarmon}) \label{p:scalarev} Let $(M^n, g_t, H_t, \phi_t)$ be a solution to generalized Ricci flow.  Then
\begin{align*}
\square R^{H,\phi} =&\ 2 \brs{\Rc^{H,\phi}}^2.
\end{align*}
\begin{proof} We compute the heat operator acting on each term of $R^{H,\phi}$ separately.  First, a standard computation (cf. \cite{GRFbook} Lemma 5.11) yields
\begin{align*}
\square R =&\ - \tfrac{1}{2} \gD \brs{H}^2 + \tfrac{1}{2} \divg \divg H^2 + 2 \IP{\Rc, \Rc - \tfrac{1}{4} H^2}.
\end{align*}
Also we have
\begin{align*}
\square \left( -\tfrac{1}{12} \brs{H}^2 \right) =&\ \IP{ \tfrac{1}{8} H^2 - \tfrac{1}{2} \Rc, H^2} - \tfrac{1}{6} \IP{\gD_d H, H} + \tfrac{1}{12} \gD \brs{H}^2.
\end{align*}
Next we have
\begin{align*}
\square \gD \phi =&\ \IP{ 2 \Rc - \tfrac{1}{2} H^2, \N^2 \phi} + \IP{ - \tfrac{1}{2} \divg H^2 + \tfrac{1}{4} d \brs{H}^2, d \phi}\\
&\ + \gD \left( \gD \phi + \tfrac{1}{6} \brs{H}^2 \right) - \gD \gD \phi\\
=&\ \IP{ 2 \Rc - \tfrac{1}{2} H^2, \N^2 \phi} + \IP{ - \tfrac{1}{2} \divg H^2 + \tfrac{1}{4} d \brs{H}^2, d \phi} + \tfrac{1}{6} \gD \brs{H}^2.
\end{align*}
Furthermore using the Bochner formula
\begin{align*}
\square \brs{\N \phi}^2 =&\ 2 \IP{\N \left( \gD \phi + \tfrac{1}{6} \brs{H}^2 \right), \N \phi} + \IP{2 \Rc - \tfrac{1}{2} H^2, d \phi \otimes d\phi} - \gD \brs{\N \phi}^2\\
=&\ - 2 \brs{\N^2 \phi}^2 + \tfrac{1}{3} \IP{ \N \brs{H}^2, \N \phi} - \tfrac{1}{2} \IP{H^2, d \phi \otimes d\phi}.
\end{align*}
Combining the above formulas and using the definition of $R^{H,\phi}$ yields
\begin{gather} \label{f:fscal10}
\begin{split}
\square R^{H,\phi} =&\ 2 \IP{\Rc, \Rc - \tfrac{1}{4} H^2} + \IP{\tfrac{1}{8} H^2 - \tfrac{1}{2} \Rc, H^2} + 2 \IP{ 2 \Rc - \tfrac{1}{2} H^2, \N^2 \phi} + 2 \brs{\N^2 \phi}^2\\
&\ + \IP{ - \divg H^2 + \tfrac{1}{2} d \brs{H}^2, d\phi} - \tfrac{1}{3} \IP{\N \brs{H}^2, \N \phi} + \tfrac{1}{2} \IP{ H^2, d\phi \otimes d\phi}\\
&\ - \tfrac{1}{12} \gD \brs{H}^2 + \tfrac{1}{2} \divg \divg H^2 - \tfrac{1}{6} \IP{\gD_d H, H}.
\end{split}
\end{gather}
We further recall the identity (cf. \cite{GRFbook} Lemma 3.19)
\begin{align*}
\divg H^2 = \tfrac{1}{6} \N \brs{H}^2 - \IP{d^*_g H, H},
\end{align*}
where $\IP{d^*_g H, H} = (d^*_g H)^{kl} H_{ikl}$.  This has the further consequence
\begin{align*}
\divg \divg H^2 = \tfrac{1}{6} \gD \brs{H}^2 + \tfrac{1}{3} \IP{ \gD_d H, H} + \brs{d^*_g H}^2.
\end{align*}
Plugging these into (\ref{f:fscal10}) yields
\begin{align*}
\square R^{H,\phi} =&\ 2 \brs{\Rc - \tfrac{1}{4} H^2 + \N^2 \phi}^2 + \tfrac{1}{2} \brs{d^*_g H + i_{\N \phi} H}^2 = 2 \brs{\Rc^{H,\phi}}^2,
\end{align*}
as claimed.
\end{proof}
\end{prop}

\begin{cor} \label{c:scalarlb} Let $(M^n, g_t, H_t, \phi_t)$ be a solution to generalized Ricci flow on a compact manifold.  Then for any smooth existence time $t$ one has
\begin{align*}
\inf_{M \times \{t\}} R^{H,\phi} \geq&\ \inf_{M \times \{0\}} R^{H,\phi}.
\end{align*}
\begin{proof} This follows from the maximum principle applied to Proposition \ref{p:scalarev}
\end{proof}
\end{cor}

\begin{cor} \label{c:rigidity2} (cf. Corollary \ref{c:rigidity}) Let $(M, g, H, \phi)$ satisfy $R^{H,\phi} \geq 0$.  If $M$ is compact, then either
\begin{enumerate}
\item The triple $(g, H, \phi)$ defines a generalized Ricci soliton and $R^{H,\phi} \equiv 0$
\item The manifold $M$ admits a triple $(\bar{g}, \bar{H}, \bar{\phi})$ such that $R^{\bar{H},\bar{\phi}} > 0$ everywhere.
\end{enumerate}
\begin{proof} This follows from the strong maximum principle applied to the evolution equation of Proposition \ref{p:scalarev}.
\end{proof}
\end{cor}

\subsection{One-form scalar curvature monotonicity} \label{s:oneform}

In the context of generalized geometry, the dilaton $\phi$, or more accurately its differential $d \phi$, plays the role of a \emph{divergence operator} \cite{GF19, GRFbook}, which is necessary to define the generalized Ricci tensor.  A natural class of divergence operators are defined by a $1$-form which need not even be closed.  Next we extend the results above to this more general setting.

\begin{defn} Given a smooth manifold $M$, a triple $(g, H, \ga)$ of a Riemannian metric, closed three-form $H$, and one-form $\ga$ determine a \emph{generalized scalar curvature}
\begin{align*}
R^{H,\ga} = R - \tfrac{1}{12} \brs{H}^2 + 2 d^* \ga - \brs{\ga}^2.
\end{align*}
\end{defn}

\begin{defn} Given $(M, g_t, H_t)$. solution to generalized Ricci flow, a one-parameter family $\ga_t$ satisfies the \emph{dilaton flow} if
\begin{gather} \label{f:1formdilatonflow}
\begin{split} 
\dt \ga =&\ \gD_d \ga + \tfrac{1}{6} d \brs{H}^2.
\end{split}
\end{gather}
\end{defn}

\begin{prop} \label{p:oneformscalarev} Let $(M^n, g_t, H_t, \ga_t)$ be a solution to generalized Ricci flow.  Then
\begin{align*}
\square R^{H,\ga} =&\ 2 \brs{\Rc - \tfrac{1}{4} H^2 + L_{\tfrac{1}{2} \ga^{\sharp}} g}^2 + \tfrac{1}{2} \brs{d \ga}^2 + \tfrac{1}{2} \brs{d^*_g H + i_{\ga^{\sharp}} H}^2.
\end{align*}
\begin{proof} The proof is nearly identical to that of Proposition \ref{p:scalarev}.  We first note
\begin{align*}
\square d^*\ga  =&\ \IP{ 2 \Rc - \tfrac{1}{2} H^2, \N \ga} + \IP{ - \tfrac{1}{2} \divg H^2 + \tfrac{1}{4} d \brs{H}^2, \ga}\\
&\ + d^* \left( \gD_d \ga + \tfrac{1}{6} d \brs{H}^2 \right) - \gD d^* \ga\\
=&\ \IP{ 2 \Rc - \tfrac{1}{2} H^2, \N \ga} + \IP{ - \tfrac{1}{2} \divg H^2 + \tfrac{1}{4} d \brs{H}^2, \ga} + \tfrac{1}{6} \gD \brs{H}^2.
\end{align*}
Furthermore using the Bochner formula
\begin{align*}
\square \brs{\ga}^2 =&\ 2 \IP{ \gD_d \ga + \tfrac{1}{6} d \brs{H}^2, \ga} + \IP{2 \Rc - \tfrac{1}{2} H^2, \ga \otimes \ga} - \gD \brs{\ga}^2\\
=&\ - 2 \brs{\N \ga}^2 + \tfrac{1}{3} \IP{ d \brs{H}^2, \ga} - \tfrac{1}{2} \IP{H^2, \ga \otimes \ga}.
\end{align*}
Using these and arguing as in Proposition \ref{p:scalarev} we obtain
\begin{align*}
\square R^{H,\ga} =&\ 2 \brs{\Rc - \tfrac{1}{4} H^2 + \N \ga}^2 + \tfrac{1}{2} \brs{d^*_g H + i_{\ga^{\sharp}} H}^2\\
=&\ 2 \brs{\Rc - \tfrac{1}{4} H^2 + L_{\tfrac{1}{2} \ga^{\sharp}} g}^2 + \tfrac{1}{2} \brs{d \ga}^2 + \tfrac{1}{2} \brs{d^*_g H + i_{\ga^{\sharp}} H}^2,
\end{align*}
as claimed.
\end{proof}
\end{prop}

\begin{rmk} An a priori lower bound for $R^{H, \ga}$ as in Corollary \ref{c:scalarlb} and a rigidity result as in Corollary \ref{c:rigidity} are immediate consequences.
\end{rmk}

\subsection{Further generalizations of Ricci flow} \label{s:GGRF}

The results of this paper extend to a more general class of geometric evolution equations coupled to differential forms.  In particular, let $H = \bigoplus_{k=1}^n H_k$ denote a closed section of $\Lambda^* T^*M$, where the subscript indicates the degree of the differential form, and consider the system of equations
\begin{gather} \label{f:GGRF}
\dt g = -2 \Rc + \tfrac{1}{2} H^2,  \qquad \dt H = \gD_d H,
\end{gather}
where as before $H^2(X, Y) = \IP{ i_X H, i_Y H}$.  We note that the constant factor of $\tfrac{1}{2}$ on the term $H^2$ may seem arbitrary, but as the coupled PDE for $H$ is linear, this constant can be tuned to any positive value, independently for each value of $k$ if desired.  This system of equations obeys the same fundamental regularity of properties as generalized Ricci flow, and has an interpretation as Ricci flow on more general Courant algebroids (cf. \cite{GRFbook}).  For this setup we define the generalized scalar curvature
\begin{align*}
R^{H,\phi} =&\ R - \tfrac{1}{4} \sum_{k=1}^n \tfrac{1}{k} \brs{H_k}^2 + 2 \gD \phi - \brs{\N \phi}^2.
\end{align*}
Also we attach a dilaton flow of the form
\begin{align} \label{f:Gdilatonflow}
\square \phi =&\ \tfrac{1}{4} \sum_{k=1}^n \tfrac{k-1}{k} \brs{H_k}^2.
\end{align}

\begin{prop} Suppose $(g_t, H_t)$ is a solution of (\ref{f:GGRF}) and $\phi_t$ is a solution of (\ref{f:Gdilatonflow}).  Then
\begin{align*}
\square R^{H,\phi} =&\ 2 \brs{\Rc - \tfrac{1}{4} H^2 + \N^2 \phi}^2 + \tfrac{1}{2} \brs{d^*_g H + i_{\N \phi} H}^2.
\end{align*}
\end{prop}

\begin{proof} For a closed differential form $H_k$ of degree $k$ one has the Bianchi identities
\begin{align*}
\divg H_k^2 =&\ - \IP{d^* H_k, H_k} + \frac{1}{2k} d \brs{H_k}^2\\
\divg \divg H_k^2 =&\ \tfrac{1}{2k} \gD \brs{H_k}^2 + \tfrac{1}{k} \IP{\gD_d H, H} + \brs{d^*_g H_k}^2.
\end{align*}
Using these a straightforward modification of the proof of Proposition \ref{p:scalarev} gives the result.
\end{proof}

\begin{rmk} An a priori lower bound for $R^{H, \phi}$ as in Corollary \ref{c:scalarlb} and a rigidity result as in Corollary \ref{c:rigidity} are immediate consequences.
\end{rmk}

\section{Weighted energy and entropy formulas} \label{s:harnack}

In \cite{Perelman1} Perelman introduced differential inequalities for the weighted scalar curvature, coupled to a solution of the conjugate heat equation.  These key estimates complement the a priori scalar curvature lower bound, and underpin the proofs of $\gk$-noncollapsing and pseudolocality for Ricci flow.  In this section, complementary to our a priori lower bound for the generalized scalar curvature, we generalize these estimates to the case of generalized Ricci flow, by further coupling to a solution of the dilaton flow.  The key point is to treat the auxiliary solution to the dilaton flow as a shift in Perelman's dilaton $f$, in particular constructing $f$ as a solution to the \emph{$\phi$-weighted} conjugate heat equation.  This leads to a family of differential inequalities with now \emph{two} auxiliary functional parameters $\phi$ and $f$.

\subsection{Weighted conjugate heat operators}

In this subsection we define the weighted conjugate heat operator, then show some elementary properties of this equation and its relation to the classic conjugate heat equation.

\begin{defn} Let $(M, g_t, H_t, \phi_t)$ denote a solution to generalized Ricci flow.  Define the \emph{conjugate heat operator}
\begin{align*}
\square^* = - \dt - \gD + R - \tfrac{1}{4} \brs{H}^2.
\end{align*}
Also, we define the \emph{weighted conjugate heat operator}
\begin{align*}
\square^*_{\phi} = - \dt - \gD + 2 \N \phi + R^{H,\phi}.
\end{align*}
\end{defn}

\begin{lemma} \label{l:conjugateheat} Let $(M, g_t, H_t, \phi_t)$ denote a solution to generalized Ricci flow.  Given $u_t, v_t$ smooth functions we have
\begin{enumerate}
\item $\frac{d}{dt} \int u v dV_g = \int_M \left(v \square u - u \square^* v \right) dV_g$.
\item $\frac{d}{dt} \int u v e^{-\phi} dV_g = \int_M \left(v \square u - u \square^*_{\phi} v \right) e^{-\phi} dV_g$.
\item A solution to $\square^* u = 0$ preserves mass against $dV_g$, i.e. $\frac{d}{dt} \int_M u dV_g = 0$.
\item A solution to $\square^*_{\phi} u = 0$ preserves mass against $e^{-\phi} dV_g$, i.e. $\frac{d}{dt} \int_M u e^{-\phi} dV_g = 0$.
\item One has $\square^* (u e^{-\phi}) = \left(\square^*_{\phi} u \right) e^{-\phi}$.
\end{enumerate}
\begin{proof} Items (1) and (3) are elementary consequences of the generalized Ricci flow equations.  To show item (2) we compute
\begin{align*}
\frac{d}{dt}& \int_M u v e^{-\phi} dV_{g}\\
=&\ \int_M \left[ v \dt u + u \dt v + u v \left( - R +\tfrac{1}{4} \brs{H}^2 - \gD \phi - \tfrac{1}{6} \brs{H}^2 \right) \right] e^{-\phi} dV_{g}\\
=&\ \int_M \left[ v \square u + u \left( e^{\phi} \gD ( e^{-\phi} v) + \dt v + v(- R +\tfrac{1}{12} \brs{H}^2 - \gD \phi) \right) \right] e^{-\phi} dV_g\\
=&\ \int_M \left[ v \square u + u \left( \dt v + \gD v - 2 \IP{\N v, \N \phi} + v \left( - R + \tfrac{1}{12} \brs{H}^2 - 2 \gD \phi + \brs{\N \phi}^2 \right) \right)\right] e^{-\phi} dV_g\\
=&\ \int_M \left[ v \square u + u  \left( \dt + \gD - 2 \N \phi - R^{H,\phi} \right) v \right] e^{-\phi} dV_g\\
=&\ \int_M \left[ v \square u - u \square^*_{\phi} v \right] e^{-\phi} dV_g,
\end{align*}
as claimed.  Item (4) is an elementary consequence of item (2).  To show item (5) we directly compute
\begin{align*}
\square^* \left( u e^{-\phi} \right) =&\ (\square^* u) e^{-\phi} - 2 \IP{\N u, \N e^{-\phi}} - u \left( \dt + \gD \right) e^{-\phi}\\
=&\ \left[ \left( - \dt - \gD + 2 \N \phi + R - \tfrac{1}{4} \brs{H}^2 \right) u \right] e^{-\phi} + u \left( \dt \phi + \gD \phi - \brs{\N \phi}^2 \right) e^{-\phi}\\
=&\ \left[ \left( - \dt - \gD + 2 \N \phi + R - \tfrac{1}{12} \brs{H}^2 + 2 \gD \phi - \brs{\N \phi}^2 \right) u \right] e^{-\phi}\\
=&\ \left( \square^*_{\phi} u \right) e^{-\phi},
\end{align*}
as claimed.
\end{proof}
\end{lemma}

\subsection{Energy density monotonicity} \label{ss:SHE}

\begin{prop} \label{p:steadyharnackev2} (cf. Proposition \ref{p:steadyharnackev}) Let $(M, g_t, H_t, \phi_t)$ denote a solution to generalized Ricci flow, and suppose $u = e^{-f}$ is a solution of $\square^*_{\phi} u = 0$.
Then
\begin{align*}
\square^*_{\phi} \left( R^{H,f+ \phi} u \right) =&\ - 2 \brs{\Rc^{H,f+\phi}}^2 u.
\end{align*}
\begin{proof} In \cite{GRFbook} Theorem 6.12 it is shown that if $v = e^{-F}$ is a solution of the conjugate heat equation $\square^* v = 0$, then
\begin{align*}
\square^* \left( R^{H,F} v \right) =&\ - 2 \brs{\Rc^{H,F}} v.
\end{align*}
It follows from Lemma \ref{l:conjugateheat} item (5) that $v = u e^{-\phi} = e^{-f-\phi}$ is a solution of $\square^* v = 0$.  Furthermore again using Lemma \ref{l:conjugateheat} item (5) we compute
\begin{align*}
\left( \square^*_{\phi} \left( R^{f+\phi} u\right) \right) e^{-\phi} = \square^* \left( R^{f+\phi} v \right) = -2 \brs{\Rc^{H,f+\phi}} u e^{-\phi},
\end{align*}
giving the claim.
\end{proof}
\end{prop}

\begin{cor} \label{c:steadyHarnack} Let $(M^n, g_t, H_t, \phi_t)$ denote a solution to generalized Ricci flow.  Suppose $u = e^{-f}$ is a solution of $\square^*_{\phi} u = 0$.  Then
\begin{align*}
\sup_{M \times \{t\}} R^{H,f + \phi} \geq \sup_{M \times \{0\}} R^{H,f + \phi}.
\end{align*}
\begin{proof} It follows from Proposition \ref{p:steadyharnackev} that $R^{f + \phi}$ is a subsolution of a backwards heat-type equation, therefore by the maximum principle its supremum is nonincreasing as a function of $-t$, therefore nondecreasing as a function of $t$.
\end{proof}
\end{cor}

\subsection{Gradient property revisited and steady solitons} \label{ss:gf}

The results in the previous subsection underpin a generalization of the gradient flow interpretation of generalized Ricci flow.
To begin we recall the usual gradient flow interpretation.  Define
\begin{align*}
\FF(g,H,f) =&\ \int_M \left( \brs{\N f}^2 + R - \tfrac{1}{12} \brs{H}^2 \right) e^{-f} dV_g,\\
\gl(g,H) =&\ \inf_{\{f\ |\ \int_M e^{-f} dV_g = 1 \}} \FF(g,H,f).
\end{align*}
Following \cite{Perelman1}, it was shown in \cite{OSW} it is shown that generalized Ricci flow is the gradient flow of $\gl$.  Furthermore, there is a general monotonicity formula for $\FF$ once $u = e^{-f}$ is imposed to solve the conjugate heat equation along the flow.  By explicitly including the dilaton shift $\phi$ we get an infinite dimensional family of eigenvalues $\gl$ for which generalized Ricci flow is the gradient flow, following Proposition \ref{p:steadyharnackev}.

\begin{defn} Given $(M^n, g, H, \phi)$, we define
\begin{align*}
\gl(g,H,\phi) = \inf_{ \{f\ |\ \int_M e^{-f} e^{-\phi} dV_g = 1 \}} \FF(g,H,f + \phi)
\end{align*}
\end{defn}

\begin{lemma} The quantity $\gl(g, H, \phi)$ is lowest eigenvalue of the operator
\begin{align*}
\mathcal L = - 4 \gD + 4 \N \phi + R^{H,\phi}.
\end{align*}
\begin{proof} Integrating by parts we can express, for $w = e^{-\frac{f}{2}}$,
\begin{align*}
\FF(g,H,f + \phi) = \int_M \left( 4 \brs{\N w}^2 + R^{H,\phi} w^2 \right) e^{-\phi} dV_g,
\end{align*}
and the result follows from a standard argument
\end{proof}
\end{lemma}

\begin{prop} \label{p:gradientflow} Generalized Ricci flow is the gradient flow of $\gl(g, H, \phi)$.
\end{prop}

Critical points for $\gl$ are steady generalized Ricci solitons, and recent constructions and classification results for these objects have appeared in \cite{ASU3, Streetssolitons, SU2, SU1}, including examples on compact manifolds.  The next proposition shows that, on such a steady soliton, the dilaton flow (suitably normalized) converges to the relevant soliton function $f$.

\begin{prop} \label{p:dilconvergence} Suppose $(M^n, g, H, f)$ is a steady generalized Ricci soliton, i.e. $\Rc^{H,f} \equiv 0$, and let $\gl = \gl(g,H)$.  Given $\phi_0 \in C^{\infty}(M)$ a smooth function, the solution to the gauge-fixed normalized dilaton flow
\begin{align*}
\left( \square + \N f \right) \phi =&\ \tfrac{1}{6} \brs{H}^2 - \gl
\end{align*}
with initial condition $\phi_0$ exists on $[0,\infty)$ and $\lim_{t \to \infty} \phi_t = f + c$ for some $c \in \mathbb R$.
\begin{proof}
We recall the basic identities for a soliton:
\begin{align*}
R - \tfrac{1}{4} \brs{H}^2 + \gD f =&\ 0,\\
R^{H,f} =&\ \gl.
\end{align*}
The first follows by tracing $\Rc^{H,f} = 0$ and the second follows by first observing that $R^f$ is constant by a Bianchi identity (cf. \cite{GRFbook} Proposition 4.33), then observing by integration against $f$, suitably normalized, that this constant must be $\gl(g,H)$.  Using these identities we compute
\begin{align*}
(\square + \N f) (\phi - f) =&\ \tfrac{1}{6} \brs{H}^2 - \gl + \gD f - \brs{\N f}^2\\
=&\ \tfrac{1}{6} \brs{H}^2 - \left( R - \tfrac{1}{12} \brs{H}^2 + 2 \gD f - \brs{\N f}^2 \right) + \gD f - \brs{\N f}^2\\
=&\ - R + \tfrac{1}{4} \brs{H}^2 - \gD f\\
=&\ 0.
\end{align*}
As $\square + \N f$ is a strictly parabolic operator with no constant term, it follows from standard results that $\phi_t$ exists for all time, and moreover that $\phi - f$ approaches a constant.
\end{proof}
\end{prop}

\subsection{Entropy density monotonicity}

\begin{defn} 
Given a smooth manifold $M$, a triple $(g, H, F)$ of a Riemannian metric, closed three-form $H$, function $F$ and $\tau > 0$ determine a \emph{generalized entropy density}
\begin{align*}
W^{H,F} = \tau R^{H,F} + F - n.
\end{align*}
\end{defn}

In the setting of Ricci flow the quantity $W^F$ is Perelman's entropy density, which satisfies a key monotonicity property.  To obtain a monotone entropy quantity for generalized Ricci flow we require a further functional parameter, which can be used to measure the concentration of $H$ at a given point.
\begin{defn} Let $(M^n, g_t, H_t, \phi_t)$ denote a solution to generalized Ricci flow, and suppose $u$ is a solution of $\square^*_{\phi} u = 0$.  A function $\psi$ is a solution of the \emph{conjugate dilaton flow} if
\begin{align*}
\square^*_{\phi} (\psi u) = - \tfrac{1}{6} \brs{H}^2 u.
\end{align*}
\end{defn}

\noindent The lemma below shows that the flow is well-posed for arbitrary smooth terminal data.

\begin{lemma} \label{l:conjdil} Given $(M^n, g_t, H_t, \phi_t)$ a solution of generalized Ricci flow, suppose $u = e^{-f}$ is a positive solution of $\square^*_{\phi} u = 0$.  Then $\psi$ is a solution of the conjugate dilaton flow if and only if
\begin{align*}
\left( - \dt - \gD + 2 \N (f + \phi) \right) \psi = - \tfrac{1}{6} \brs{H}^2.
\end{align*}
\begin{proof} Given a smooth function $\psi$ we formally compute
\begin{align*}
\square^*_{\phi} (\psi u) =&\ \left( \left( -\dt - \gD + 2 \N \phi \right) \psi \right) u + \psi \square^*_{\phi} u - 2 \IP{\N \psi, \N u}\\
=&\ \left( \left( -\dt - \gD + 2 \N (f + \phi) \right) \psi \right) u.
\end{align*}
Using that $u$ is positive, the result follows.
\end{proof}
\end{lemma}

\begin{prop} \label{p:Ricciflowharnack} (cf. Proposition \ref{p:GRFharnack}) Let $(M, g_t, H_t, \phi_t)$ denote a solution to generalized Ricci flow, and suppose $u = (4 \pi \tau)^{-\frac{n}{2}} e^{-f}$ is a solution of $\square^*_{\phi} u = 0$, where $\tau = T - t$ for some fixed $T$.  Furthermore suppose $\psi$ is a solution of the conjugate dilaton flow.  Then
\begin{align*}
\square^*_{\phi} \left[ (W^{H,f+\phi} + \psi) u\right] =&\ - 2 \tau \brs{\Rc^{H,f+\phi} - \frac{1}{2 \tau} g}^2 u.
\end{align*}
\begin{proof} First observe, using $\square^*_{\phi} u = 0$,
\begin{gather} \label{f:shrinkerfev}
\begin{split}
0 =&\ \square^*_{\phi} u = \left( - \dt - \gD + 2 \N \phi + R^{H,\phi} \right) \left( (4 \pi \tau)^{-\frac{n}{2}} e^{-f} \right)\\
=&\ \left( - \frac{n}{2 \tau} + \frac{\del f}{\del t} + \gD f - \brs{\N f}^2 - 2 \IP{\N \phi, \N f} + R^{H,\phi} \right) \left( (4 \pi \tau)^{-\frac{n}{2}} e^{-f} \right).
\end{split}
\end{gather}
Now we compute using Proposition \ref{p:steadyharnackev},
\begin{align*}
& \left(- \dt - \gD + 2 \N \phi \right) W^{H,f+\phi}\\
& \qquad = R^{H,f + \phi} + \tau \left(- \dt - \gD + 2 \N \phi\right) R^{H,f + \phi} + \left(-\dt - \gD + 2 \N \phi \right) \left(f + \phi \right)\\
& \qquad = \left( 2 \gD (f + \phi ) - \brs{\N f + \phi}^2 + R - \tfrac{1}{12} \brs{H}^2 \right)\\
&\ \qquad \qquad + \tau \left( - 2 \brs{\Rc - \tfrac{1}{4} H^2 + \N^2 (f + \phi)}^2 - \tfrac{1}{2} \brs{d^*_g H + i_{\N (f + \phi)} H}^2 - 2 \IP{\N R^{H,f+\phi}, \N f} \right)\\
&\ \qquad \qquad + \left( -\frac{n}{2 \tau} - \brs{\N f}^2 + 2 \gD \phi - \brs{\N \phi}^2 + R -\tfrac{1}{12} \brs{H}^2 \right) + \left( - 2 \gD \phi + 2 \brs{\N \phi}^2 - \tfrac{1}{6} \brs{H}^2 \right)\\
& \qquad = - 2 \tau \brs{\Rc -\tfrac{1}{4} H^2 + \N^2(f + \phi) - \frac{1}{2 \tau} g}^2 - \tfrac{\tau}{2} \brs{d^*_g H + i_{\N (f + \phi)} H}^2  + \tfrac{1}{6} \brs{H}^2\\
&\ \qquad \qquad - 2 \IP{\N (\tau R^{f + \phi}), \N f} - 2 \brs{\N f}^2 - 2 \IP{\N f, \N \phi}\\
& \qquad = - 2 \tau \brs{\Rc^{H,f+\phi} - \frac{1}{2 \tau} g}^2 + \tfrac{1}{6} \brs{H}^2 - 2 \IP{\N W^{H,f+\phi} , \N f}.
\end{align*}
We then obtain
\begin{align*}
\square^*_{\phi} \left( W^{H,f+\phi} u \right) =&\ \left( - \dt - \gD + 2 \N \phi + R^{H,\phi} \right) (W^{H,f+\phi} u)\\
=&\ \left( \left(- \dt - \gD + 2 \N \phi \right) W^{H,f+\phi} \right) u + W^{H,f+\phi} \square^*_{\phi} u - 2 \IP{\N W^{H,f+\phi}, \N u}\\
=&\ - 2 \tau \brs{\Rc^{H,f+\phi} - \frac{1}{2 \tau} g}^2 u  + \tfrac{1}{6} \brs{H}^2 u,
\end{align*}
and the result follows.
\end{proof}
\end{prop}

\subsection{Perelman's Harnack estimate}

To prove the generalization of Perelman's Harnack estimate we require some technical estimates for the conjugate heat kernel along solutions to generalized Ricci flow, which follow from straightforward modifications of results in \cite{ChauPseudo} (cf. also \cite{NiLYH, Zhang1}).  We fix $(M^n, g_t, H_t, \phi_t)$ a solution of generalized Ricci flow on $M \times [0,T]$, and let $Z(x,t,y,s)$ denote the fundamental solution of $\square u = 0$.  We fix $p \in M$ and define $u(x,t) = Z(p,T,x,t)$.  It follows that $u$ is a solution of the conjugate heat equation $\square^* u = 0$.  We define $f$ by $u =  (4 \pi \tau)^{- \frac{n}{2}} e^{-f}$ as above, where $\tau = T-  t$.  Furthermore, we fix $0 < t_0 < T$ and let $h_{t_0} \geq 0$ denote a smooth compactly supported function, and let $h_t$ denote the solution to $\square h = 0$ on $[t_0, T]$ with initial condition $h_{t_0}$.  We then claim the following fundamental characteristics for the entropy density $W^{H,f}$.

\begin{prop} \label{p:heatkernel} (cf. \cite{ChauPseudo} Theorem 7.1) Given the setup above,
\begin{enumerate}
\item For any $t_0 < t < T$ one has that $h W^{H,f} \in L^1(M, g_t)$.
\item For any $t_0 < t_1 < t_2 < T$ one has
\begin{align*}
\int_M h W^{H,f} dV_{g_{t_1}} \leq \int_M h W^{H,f} dV_{g_{t_2}} + \int_{t_1}^{t_2} \int_M \tfrac{1}{6} h \brs{H}^2 u dV_{g_t} dt.
\end{align*}
\item One has
\begin{align*}
\limsup_{t \to T^-} \int_M h W^{H,f} dV_{g_t} \leq 0.
\end{align*}
\end{enumerate}
\end{prop}

\begin{cor} \label{c:GRFHarnack} Let $(M^n, g_t, H_t, \phi_t)$ denote a solution to generalized Ricci flow on a compact manifold, defined for $t \in [0,T]$.  Suppose $u = (4 \pi (T - t))^{-\frac{n}{2}} e^{-f}$ approaches $e^{\phi} \gd_{x_0}$ as $t \to T$ and solves $\square^*_{\phi} u = 0$.  Finally let $\psi$ denote the unique solution of the conjugate dilaton flow associated to $u$ with $\psi_T \equiv 0$.  Then for any $0 \leq t_0 \leq T$ one has
\begin{align*}
\left( W^{H,f+\phi} + \psi \right) u \leq 0.
\end{align*}
\begin{proof} Suppose there exists a point $(y_0,t_0)$ such that
\begin{align*}
V := \left( W^{H,f+\phi} + \psi \right) u (y_0,t_0) > 0.
\end{align*}
Choose a smooth nonnnegative function $h$ such that $h(y_0) = 1$ and $\sup_M h \leq 1$.  Solve the forward heat equation $\square h = 0$ on $[t_0,T]$ with initial data $h$ at time $t_0$.  Note that by the maximum principle $h$ remains nonnegative.  We then compute using Lemma \ref{l:conjugateheat} and Proposition \ref{p:Ricciflowharnack},
\begin{align*}
\frac{d}{dt} \int_M h V e^{-\phi} dV_g  =&\ - \int_M h \square^*_{\phi} V e^{-\phi} dV_g \geq 0.
\end{align*}
However we obtain a contradiction since by construction $\int_M h V e^{-\phi} dV_{g_{t_0}} > 0$ and by Proposition \ref{p:heatkernel} and the fact that $\psi_T \equiv 0$,
\begin{align*}
\lim_{t \to T} \int_M h V e^{-\phi} dV_g = 0.
\end{align*}
\end{proof}
\end{cor}

\section{Weighted Nash entropy and pseudolocality estimates}

In this section we give an extension of the Nash entropy convexity for Ricci flow to the setting of weighted scalar curvature.  Then we derive a pseudolocality principle for Ricci flow in terms of weighted scalar curvature generalizing Perelman's result \cite{Perelman1}.  The key new technical point is a weighted isoperimetric inequality and its relationship to a weighted log-Sobolev inequality.

\subsection{Weighted Nash and Perelman Entropies}

\begin{defn} Given a smooth Riemannian manifold $(M, g)$, $f, \phi \in C^{\infty}(M)$ and $\tau > 0$, let $d \nu = (4 \pi \tau)^{-\frac{n}{2}} e^{-f} e^{-\phi} dV_g$, and define the \emph{weighted Nash entropy} by
\begin{gather} \label{f:nash}
\NN (g, \phi, f, \tau) := \int_M  f  d \nu - \frac{n}{2}.
\end{gather}
Furthermore define the \emph{weighted Perelman entropy} by
\begin{gather} \label{f:Perelmanent}
\WW (g, \phi, f, \tau) := \int_M \left[ \tau R^{f + \phi} + f - n \right] d \nu.
\end{gather}
\end{defn}

\begin{prop} \label{p:weightednash} Let $(M, g_t, \phi_t)$ denote a solution to Ricci flow defined on $(-T, 0]$, and suppose $u_t$ is a solution of the conjugate heat equation $\square^*_{\phi} u = 0$, and define $f$ via $u = (4 \pi \tau)^{-\frac{n}{2}} e^{-f}$, where $\tau = -t$, and suppose the measure $d \nu$ has unit mass.  Then
\begin{align*}
\frac{d}{d\tau} \left( \tau \NN \right) =&\ \WW, \qquad \frac{d}{d\tau} \WW = -2 \tau \int_M \brs{\Rc^{f + \phi} - \frac{1}{2 \tau} g}^2 d \nu.
\end{align*}
\begin{proof} Since $\square^*_{\phi} u = 0$ an elementary computation shows that we have that
\begin{align*}
\frac{d}{d \tau} d\nu =&\ \left( u^{-1} \gD u - 2 u^{-1} \IP{\N \phi, \N u} - \gD \phi + \brs{\N \phi}^2 \right) d \nu.
\end{align*}
Furthermore, using the evolution equation (\ref{f:shrinkerfev}) satisfied by f and integration by parts we compute 
\begin{align*}
\frac{d}{d\tau} \int_M f d \nu =&\ \int_M \left( \gD f - \brs{\N f}^2 - 2 \IP{\N \phi, \N f} + R^{\phi} - \frac{n}{2 \tau} \right) d \nu\\
&\ + \int_M f \left( u^{-1} \gD u - 2 u^{-1} \IP{\N \phi, \N u} - \gD \phi + \brs{\N \phi}^2 \right) d \nu\\
=&\ \int_M \left( R^{f + \phi} - \frac{n}{2 \tau} \right) d \nu.
\end{align*}
Thus, using that $\int_M d \nu \equiv 1$, we obtain
\begin{align*}
\frac{d}{d \tau} \left( \tau \NN \right) =&\ \int_M \left[ \tau R^{f + \phi} + f - n\right] d \nu = \WW,
\end{align*}
as claimed.  Furthermore we note that we can express
\begin{align*}
\WW = \int_M \left[ \WW^{f + \phi} - \phi \right] d \nu.
\end{align*}
Then it follows by Lemma \ref{l:conjugateheat} and Proposition \ref{p:Ricciflowharnack} with $\psi \equiv 0$ that
\begin{align*}
\frac{d}{d\tau} \WW =&\ \frac{d}{d\tau} \int_M \left[  \left( \WW^{f + \phi} - \phi \right) u \right] e^{-\phi} dV_g = \int_M \left[ \square^*_{\phi} (\WW^{f + \phi} u) + u \square \phi - \phi \square_{\phi}^* u \right] e^{-\phi} dV_g\\
=&\ - 2 \tau \int_M \brs{\Rc^{f + \phi} - \frac{1}{2 \tau} g}^2 d \nu,
\end{align*}
as claimed.
\end{proof}
\end{prop}

\subsection{Weighted log-Sobolev and isoperimetric inequalities}

\begin{defn} \label{d:wLS} Given a Riemannian manifold $(M^n, g)$ and $\phi \in C^{\infty}(M)$, we say that this data satisfies a \emph{$\phi$-weighted Sobolev inequality} if there exists $\gL = \gL(g, \phi) > -\infty$ such that for all $u = (2 \pi)^{-\frac{n}{2}} e^{-f}$ such that $\int_M u e^{-\phi} dV_g = 1$ one has
\begin{align} \label{f:WLS}
\int_M \left( \tfrac{1}{2} \brs{\N f}^2 + f - n \right) u e^{-\phi} dV_g \geq \gL.
\end{align}
\end{defn}

\begin{rmk} Given $\psi \in W^{1,2}(M)$, define $f$ by
\begin{align*}
\psi^2 = (2 \pi)^{-\frac{n}{2}} e^{-f} \int_M \psi^2 e^{-\phi} dV_g = u \int_M \psi^2 e^{-\phi} dV_g.
\end{align*}
It then follows that (\ref{f:WLS}) is equivalent to
\begin{align*}
\int_M & \left( 2 \brs{\N \psi}^2 - \psi^2 \log \psi^2 \right) e^{-\phi} dV_g + \log \left( \int_M \psi^2 e^{-\phi} dV_g \right) \int_M \psi^2 e^{-\phi} dV_g\\
\geq&\ \left( \frac{n}{2} \log (2 \pi) + n + \gL \right) \int_M \psi^2 e^{-\phi} dV_g.
\end{align*}
\end{rmk}

\begin{defn} \label{d:wIso} Fix $(M, g)$ a Riemannian manifold and $\phi$ a smooth function.  We say that $(M, g)$ satisfies a \emph{$\phi$-weighted isoperimetric inequality} with constant $I_n^{\phi}$ if for all compact domains $\Omega \subset M$ with $C^1$ boundary one has
\begin{align*}
\left [ \Area_{e^{-\phi} dA} (\del \Omega) \right]^n \geq I_n^{\phi} \left[ \Vol_{e^{-\phi} dV_g} (\Omega) \right]^{n-1}.
\end{align*}
\end{defn}

\begin{rmk} \label{r:weightediso} The $\phi$-weighted isoperimetric inequality is scale-invariant for the metric $g$, whereas for a constant $\gl$ an elementary argument shows that $I_n^{\phi + \gl} = e^{-\gl} I_n^{\phi}$.
\end{rmk}

\begin{thm} \label{t:isotosob} Suppose $(M^n, g)$ is a Riemannian manifold and fix $\phi$ a smooth function.  Suppose $\bar{B}(x_0, R)$ is compact, and that $(B(x_0, R), g)$ satisfies a $\phi$-weighted isoperimetric inequality with constant $I_n^{\phi}$.  Then for any $C^1$ function $\psi$ compactly supported in $B(x_0, R)$ one has
\begin{align*}
\int_M & \left( 2 \brs{\N \psi}^2 - \psi^2 \log \psi^2 \right) e^{-\phi} dV_g + \log \left( \int_M \psi^2 e^{-\phi} dV_g \right) \int_M \psi^2 e^{-\phi} dV_g\\
\geq&\ \left( \frac{n}{2} \log (2 \pi) + n + \log \left(\frac{I_n^{\phi}}{c_n} \right) \right) \int_M \psi^2 e^{-\phi} dV_g,
\end{align*}
where $c_n$ denotes the Euclidean isoperimetric constant.  In particular, $(B(x_0, R), g)$ satisfies a $\phi$-weighted Sobolev inequality with constant
\begin{align} \label{f:Sobolevconst}
\gL = \log \left( \frac{I_n^{\phi}}{c_n} \right).
\end{align}
\begin{proof} For simplicity of notation we can just set $M = B(x_0, R)$.  Also we perform a useful rescaling.  If we set $\til{g} = \left( \frac{c_n}{I_n^{\phi}} \right)^{\frac{2}{n}} g$ then the required inequality is equivalent to
\begin{gather} \label{f:sobiso10}
\begin{split}
\int_M & \left( 2 \left( \frac{c_n}{I_n^{\phi}} \right)^{\frac{2}{n}} \brs{\N \psi}^2 - \psi^2 \log \psi^2 \right) e^{-\phi} dV_{\til{g}} + \log \left( \int_M \psi^2 e^{-\phi} dV_{\til{g}} \right) \int_M \psi^2 e^{-\phi} dV_{\til{g}}\\
\geq&\ \left( \frac{n}{2} \log (2 \pi) + n \right) \int_M \psi^2 e^{-\phi} dV_{\til{g}}.
\end{split}
\end{gather}
We will show (\ref{f:sobiso10}), dropping the tildes from the notation for simplicity.  By an approximation argument it suffices to consider the case $\psi \geq 0$.  For $s > 0$ let $M_s = \{x \in M\ |\ \psi \geq s \}$.  Let $\gG_s = \del M_s$, and let
\begin{align*}
F(s) := \Vol_{e^{-\phi} dV_g}(M_s).
\end{align*}
Choose $r_0 < \infty$ such that
\begin{align*}
\Vol_{e^{-\phi} dV_g} ( \{ \psi > 0 \} ) = \gw_n r_0^n = \Vol_{dV_{\mathbb R^n}} (B_{r_0}).
\end{align*}
Choose the unique rotationally invariant function $h : \mathbb R^n \to \mathbb R$ such that
\begin{align*}
\Vol_{dV_{\mathbb R^n}} ( \{ h \geq s \}) = F(s),
\end{align*}
and such that $h(x) = 0$ for $\brs{x} \geq r_0$.  Furthermore set $M'_s = \{ h \geq s \}$, and $\gG'_s = \del M'_s$.  
We recall the statement of the co-area formula,
\begin{align*}
\int_M H \brs{\N f} dV_g = \int_{-\infty}^{\infty} \int_{\{f = s\}} H dA ds.
\end{align*}
Applying this with $H = e^{-\phi} \brs{\N \psi}^{-1}$ and $f = \psi$ on $M_t$ yields
\begin{align*}
F(t) = \Vol_{e^{-\phi} dV_g} (M_t) = \int_t^{\infty} \int_{\{\psi = s\}} \brs{\N \psi}^{-1} e^{-\phi} dA ds.
\end{align*}
This implies that for almost all $s$ one has
\begin{align*}
- \frac{dF}{ds}(s) = \int_{\{\psi = s\}} \brs{\N \psi}^{-1} e^{-\phi} dA.
\end{align*}
Thus for an arbitrary Lipschitz function $\gl$ we have
\begin{align*}
- \int_0^{\infty} \gl \frac{d F}{ds} ds = \int_0^{\infty} \gl \int_{ \{\psi = s\}} \brs{ \N \psi}^{-1} e^{-\phi} dA ds.
\end{align*}
On the other hand, applying the co-area formula with $H = \gl(\psi) \brs{\N \psi}^{-1} e^{-\phi}$ yields
\begin{align*}
\int_{ \{\psi > 0\}} \gl(\psi) e^{-\phi} dV_g = \int_0^{\infty} \gl(s) \int_{\{\psi = s\}} \brs{\N \psi}^{-1} e^{-\phi} dA ds.
\end{align*}
Integrating by parts we have
\begin{align*}
\int_0^{\infty} \frac{d \gl}{ds} F ds = \int_{\{\psi > 0\}} \gl(\psi) e^{-\phi} dV_g.
\end{align*}
By construction, we may argue similary with the function $\gl(h)$ on $\mathbb R^n$ to obtain
\begin{align*}
\int_{\{h > 0\}} \gl(h) dV_{\mathbb R^n} = \int_0^{\infty} \frac{d \gl}{ds} F(s) = \int_{\{\psi > 0\}} \gl(\psi) e^{-\phi} dV_g.
\end{align*}
Choosing $\gl(s) = (\log s^2) s^2$ and $\gl(s) = s^2$ thus implies
\begin{gather} \label{f:sobiso20}
\begin{split}
\int_M ( \log \psi^2) \psi^2 ) dV_g =&\ \int_{\mathbb R^n} ( \log h^2) h^2 dV_{\mathbb R^n}, \qquad \int_M \psi^2 dV_g = \int_{\mathbb R^n} h^2 dV_{\mathbb R^n}.
\end{split}
\end{gather}
This reduces the result to comparing to comparing the gradient terms.  By the co-area formula we have
\begin{align*}
\int_t^{\infty} \int_{\gG_s} \brs{\N \psi}^{-1} e^{-\phi} dA_g ds = F(t) = \int_{t}^{\infty} \int_{\gG'_s} \brs{\N h}^{-1} dA_{\mathbb R^n} ds.
\end{align*}
Differentiating this yields
\begin{align*}
\int_{\gG_s} \brs{\N \psi}^{-1} e^{-\phi} dA = \int_{\gG'_s} \brs{\N h}^{-1} dA_{\mathbb R^n}.
\end{align*}
Furthermore note that by construction one has
\begin{align*}
(\Area( \gG'_s))^n =&\ c_n (\Vol (M'_s))^{n-1} = c_n(\Vol_{e^{-\phi} dV_g} (M_s))^{n-1}\\
\leq&\ \frac{c_n}{I_n^{\phi}} (\Area_{e^{-\phi} dA}(\gG_s))^n.
\end{align*}
Furthermore, using that $h$ is rotationally symmetric, and applying H\"older's inequality we get
\begin{align*}
\int_{\gG'_s} \brs{\N h} dA_{\mathbb R^n} \cdot \int_{\gG'_s} \brs{\N h}^{-1} dA_{\mathbb R^n} =&\ (\Area(\gG'_s))^2\\
\leq&\ \left( \frac{c_n}{I_n^{\phi}} \right)^{\frac{2}{n}} (\Area_{e^{-\phi} dA} (\gG_s))^2\\
\leq&\ \left( \frac{c_n}{I_n^{\phi}} \right)^{\frac{2}{n}} \int_{\gG_s} \brs{\N \psi} e^{-\phi} dA_g \cdot \int_{\gG_s} \brs{\N \psi}^{-1} e^{-\phi} dA_g.
\end{align*}
This then implies
\begin{align*}
\int_{\gG'_s} \brs{\N h} dA_{\mathbb R^n} \leq \left( \frac{c_n}{I_n^{\phi}} \right)^{\frac{2}{n}} \int_{\gG_s} \brs{\N \psi} e^{-\phi} dA_g.
\end{align*}
The co-area formula then implies
\begin{gather} \label{f:sobiso30}
\begin{split}
\left( \frac{c_n}{I_n^{\phi}} \right)^{\frac{2}{n}} \int_M \brs{\N \psi}^2 e^{-\phi} dV_g =&\ \left( \frac{c_n}{I_n^{\phi}} \right)^{\frac{2}{n}} \int_0^{\infty} \int_{\gG_s} \brs{\N \psi} e^{-\phi} dA_{g} ds\\
\geq&\  \int_0^{\infty} \int_{\gG'_s} \brs{\N h} dA_{\mathbb R^n}\\
=&\ \int_{\mathbb R^n} \brs{\N h}^2 dV_{\mathbb R^n}.
\end{split}
\end{gather}
Combining (\ref{f:sobiso20}), (\ref{f:sobiso30}) and applying the Euclidean logarithmic Sobolev inequality \cite{Gross} yields
\begin{align*}
\int_M & \left( 2 \left( \frac{c_n}{I_n^{\phi}} \right)^{\frac{2}{n}} \brs{\N \psi}^2 - \psi^2 \log \psi^2 \right) e^{-\phi} dV_{{g}} + \log \left( \int_M \psi^2 e^{-\phi} dV_{g} \right) \int_M \psi^2 e^{-\phi} dV_{{g}}\\
&\ - \left( \frac{n}{2} \log (2 \pi) + n \right) \int_M \psi^2 e^{-\phi} dV_{{g}}\\
\geq&\ \int_{\mathbb R^n} \left( 2 \brs{\N h}^2 - h^2 \log h^2 \right) dV_{\mathbb R^n} + \log \left( \int_{\mathbb R^n} h^2 dV_{\mathbb R^n} \right) \int_{\mathbb R^n} h^2 dV_{\mathbb R^n}\\
&\ - \left( \frac{n}{2} \log (2 \pi) + n \right) \int_{\mathbb R^n} h^2 dV_{\mathbb R^n}\\
\geq&\ 0,
\end{align*}
as required.
\end{proof}
\end{thm}

\subsection{Weighted pseudolocality estimate}

\begin{thm} \label{t:RFgenpseudo2} (cf. Theorem \ref{t:RFgenpseudo}) For every $\ga > 0$ there exist $\gd, \ge > 0$ satisfying the following: Suppose we have a smooth pointed Ricci flow solution $(M, (p_0, 0), g_t)$ defined for $t \in [0, (\ge r_0)^2]$, such that each time slice is complete.  Suppose that there exists $\phi_0 \in C^{\infty}_0(M)$ such that
\begin{enumerate}
\item $R^{\phi_0}(p, 0) \geq - r_0^{-2}$ for any $p \in B_0(p_0, r_0)$,
\item $\phi_0(p, 0) \geq - \gd$ for any $p \in B_0(p_0, r_0)$,
\item The $\phi_0$-weighted isoperimetric constant of $B_0(p_0, r_0)$ satisfies $I_n^{\phi_0} \geq (1 - \gd ) c_n$, where $c_n$ denotes the Euclidean isoperimetric constant.
\end{enumerate}
Then $\brs{\Rm}(p,t) < \ga t^{-1} + (\ge r_0)^{-2}$ whenever $0 < t \leq (\ge r_0)^2$ and $d_t(p,p_0) \leq \ge r_0$.
\end{thm}

\begin{rmk} In the hypotheses above there is both a lower bound for $\phi_0$ in item (2) and an implicit upper bound for $\phi_0$ in item (3) (cf. Remark \ref{r:weightediso}).
\end{rmk}

\begin{proof} We follow the proof in \cite{KL} and briefly indicate the initial phases.  By scale invariance of the hypotheses and conclusion it suffices to consider the case $r_0 = 1$ and also $\ga < \tfrac{1}{100n}$.  If the theorem is false, there exist sequences $\ge_k \to 0, \gd_k \to 0$ pointed Ricci flow solutions $(M_k, (p_{0,k},0), g^k)$ and functions $\phi_0^k$ which satisfy the hypotheses of the theorem but for which there exists $(p_k, t_k)$ with $0 < t_k \leq \ge_k$, $d(p_k, t_k) \leq \ge_k$, but $\brs{\Rm}(p_k, t_k) \geq \ga t_k^{-1} + \ge_k^{-2}$.  Choose $A_k = \tfrac{1}{100 n \ge_k}$, and employ (\cite{KL} Lemma 31.1) to obtain a new sequence $(\bar{p}_k, \bar{t}_k)$ which lie at the center of parabolic balls of a controlled size.  We let $u_k = (4 \pi (\bar{t}_k - t))^{-\frac{n}{2}} e^{-F_k}$ satisfy $\square^* u_k = 0$ and $\lim_{t \to \bar{t}_k^-} u(p, t) = \gd_{\bar{p}_k}(p)$.  Furthermore let $W^k = W^{F_k}$.  It follows from (\cite{KL} Lemma 33.4) that there exists $\gb > 0$ so that for all sufficiently large $k$, there exists $\til{t}_k \in [\bar{t}_k - \tfrac{1}{2} \ga Q_k^{-1}, \bar{t}_k]$ with $\int_{B_k} W^k dV_k \leq - \gb$, where $Q_k = \brs{\Rm}(\bar{p}_k, \bar{t}_k)$, and $B_k = B_{\til{t}_k}(\bar{p}_k, \sqrt{\bar{t}_k - \til{t}_k})$.

The aim is to use this to contradict the weighted log-Sobolev inequality at the initial time.  We drop the subscripts and work with a particular large $k$.  One constructs a function $h$ on spacetime from a modified distance function satisfying certain estimates we will use below (cf. \cite{KL} \S 34).  Using Perelman's Harnack estimate and estimates for $h$ we produce a function $u = (4 \pi \bar{t})^{-\frac{n}{2}} e^{-F}$ at time zero with mass arbitrarily close to $1$ which satisfies (cf. \cite{KL} (34.6))
\begin{align*}
\gb(1 - A^{-2}) \leq - \int W^F  h u dV_g.
\end{align*}
We drop the subscript on $\phi_0$ and define $f$ via $F = f + \phi$.  We thus express
\begin{gather} \label{f:pseudo10}
\begin{split}
\gb (1 - A^{-2}) \leq&\ \int_M \left[ \left(-2 \gD \phi - 2 \gD f + \brs{\N (f + \phi)}^2 - R \right) \bar{t} - (f + \phi) + n \right] h u dV_g\\
=&\ \int_M \left[ \left(- R^{\phi} - 2 \gD f + \brs{\N f}^2 + 2 \IP{\N f, \N \phi}  \right) \bar{t} - (f + \phi) + n \right] h u dV_g.
\end{split}
\end{gather}
We set $\til{f} = f - \log h$ and integrate by parts to obtain
\begin{align*}
\int_M & \left( -2 \gD f + \brs{\N f}^2 + 2 \IP{\N f, \N \phi} \right) h e^{-f - \phi} dV_g\\
&\ \quad = \int_M \left( 2 \IP{\N f, \N (h e^{-f - \phi})} + (\brs{\N f}^2 + 2 \IP{\N f, \N \phi}) h e^{-f-\phi} \right) dV_g\\
&\ \quad = \int_M  \IP{\N f, 2 h^{-1} \N h - \N f} h e^{-f - \phi} dV_g\\
&\ \quad = \int_M \left( - \brs{\N \til{f}}^2 + h^{-2} \brs{\N h}^2 \right) h e^{-f - \phi} dV_g.
\end{align*}
Furthermore, by hypothesis $- R^{\phi} h \leq 1$ and by construction $h^{-1} \brs{\N h}^2 \leq \frac{10}{(10 A \ge)^2}$, so that
\begin{align*}
\int_M \bar{t} \left( \frac{\brs{\N h}^2}{h} - R^{\phi} h \right) u dV_g \leq A^{-2} + \ge^2.
\end{align*}
Also by construction we have
\begin{align*}
- \int_M u h \log h \leq&\ c A^{-2}
\end{align*}
for a controlled constant $c$.  Finally, by hypothesis (2) we note that
\begin{align*}
\int_M - \phi h u \leq \gd.
\end{align*}
Plugging the above observations into (\ref{f:pseudo10}) yields
\begin{align*}
\gb(1 - A^{-2}) - (1 + c) A^{-2} - \ge^2 - \gd \leq \int_M \left( - \bar{t} \brs{\N \til{f}}^2 - \til{f} + n \right) (4 \pi \bar{t})^{-\frac{n}{2}} e^{-\til{f}} e^{-\phi} dV_g.
\end{align*}
For $\ge, \gd$ sufficiently small it follows then that
\begin{align*}
\tfrac{1}{2} \gb \leq \int_M \left( - \bar{t} \brs{\N \til{f}}^2 - \til{f} + n \right) (4 \pi \bar{t})^{-\frac{n}{2}} e^{-\til{f}} e^{-\phi} dV_g.
\end{align*}
For $\gd$ chosen sufficiently small this contradicts hypothesis (3) using Theorem \ref{t:isotosob}.
\end{proof}

\begin{rmk} \label{r:diffeo} Theorem \ref{t:RFgenpseudo} implies via smoothing a diffeomorphism finiteness result for the class of manifolds with a weighted scalar curvature lower bound, almost nonnegative weight, volume upper bound, and almost-Euclidean weighted isoperimetric inequality (cf. \cite{KL} Theorem 37.1).
\end{rmk}



\end{document}